\documentclass[a4paper]{amsart}
\usepackage[latin1]{inputenc}
\usepackage{ae,aecompl,amsbsy,amssymb,amsmath,amsthm,
eurosym,amsfonts,epsfig,graphicx,graphics,verbatim,enumerate,esint,MnSymbol}
\usepackage[usenames, dvipsnames]{color}
\usepackage{hyperref}

\newcommand{\quotes}[1]{\lq#1\rq}

\newcommand{\norm}[1]{\lVert #1 \rVert}

\newcommand{\br}{\mathbb{R}}

\newcommand{\cd}{\mathcal{D}}

\DeclareMathOperator{\summation}{sum}

\newcommand{\abs}[1]{\lvert#1\rvert}

\newcommand{\dsigma}{\mathrm{d}\sigma}
\newcommand{\domega}{\mathrm{d}\omega}
\newcommand{\dmu}{\mathrm{d}\mu}

\newcommand{\cdotroomy}{\,\cdot\,}

\newcommand{\cq}{\mathcal{Q}}

\newcommand{\angles}[1]{\langle #1 \rangle}

\usepackage{enumitem}

\theoremstyle{plain}
\newtheorem{theorem}{Theorem}[section]
\newtheorem{lemma}[theorem]{Lemma}
\newtheorem{proposition}[theorem]{Proposition}
\newtheorem{corollary}[theorem]{Corollary}

\theoremstyle{definition}

\theoremstyle{remark}
\newtheorem*{remark}{Remark}

\numberwithin{equation}{section}

\begin{document}
\title[Two-weight norm inequalities]{On two-weight norm inequalities for positive dyadic operators}

\author{Timo S. H\"anninen}
\address{Department of Mathematics and Statistics, University of Helsinki, P.O. Box 68, FI-00014 HELSINKI, FINLAND}
\email{timo.s.hanninen@helsinki.fi}
\author{Igor E. Verbitsky}
\address{Department of Mathematics, University of Missouri, Columbia, MO  65211, USA}
\email{verbitskyi@missouri.edu}

\thanks{T.S.H. is supported by the Academy of Finland through funding of his postdoctoral researcher post (Funding Decision No 297929), and by the Jenny and Antti Wihuri Foundation through covering of expenses of his visit to the University of Missouri. He is a member of the Finnish Centre of Excellence in Analysis and Dynamics Research.}

\begin{abstract} Let $\sigma$ and $\omega$ be locally finite Borel measures on $\br^d$, and let $p\in(1,\infty)$ and $q\in(0,\infty)$. We study the two-weight norm inequality
$$
\lVert T(f\sigma) \rVert_{L^q(\omega)}\leq C \lVert f \rVert_{L^p(\sigma)}, \quad \text{for all} \, \,  f \in L^p(\sigma), 
$$
for both the positive summation operators $T=T_\lambda(\cdotroomy \sigma)$ and positive maximal operators $T=M_\lambda(\cdotroomy \sigma)$. Here, for a family $\{\lambda_Q\}$ of non-negative reals indexed by the dyadic cubes $Q$, these operators are defined by
$$
T_\lambda(f\sigma):=\sum_Q \lambda_Q \angles{f}^\sigma_Q 1_Q \quad\text{ and } \quad M_\lambda(f\sigma):=\sup_Q \lambda_Q \angles{f}^\sigma_Q 1_Q,
$$
where $\angles{f}^\sigma_Q:=\frac{1}{\sigma(Q)} \int_Q |f| d \sigma.$
We obtain new characterizations of the two-weight norm inequalities in the following cases:
\begin{itemize}
\item For  $T=T_\lambda(\cdotroomy \sigma)$ in the subrange $q<p$. Under the additional assumption that $\sigma$ satisfies  the $A_\infty$ condition with respect to $\omega$, we characterize the inequality in terms of a simple integral condition. The proof is based on characterizing the multipliers between certain classes of Carleson measures.

\item For $T=M_\lambda(\cdotroomy \sigma)$ in the subrange $q<p$. We introduce a scale of simple conditions that depends on an integrability parameter and show that, on this scale, the sufficiency and necessity are separated only by an arbitrarily small integrability gap.

\item For the summation operators $T=T_\lambda(\cdotroomy \sigma)$ in the subrange $1<q<p$. We characterize the inequality for summation operators by means of related inequalities for maximal operators $T=M_\lambda(\cdotroomy \sigma)$. This maximal-type characterization is an alternative to the known  potential-type characterization.
\end{itemize} 
The subrange of the exponents $q<p$ 
appeared recently in applications to nonlinear elliptic PDE with $\lambda_Q = \sigma(Q) \, |Q|^{\frac{\alpha}{d}-1}$, $\alpha \in (0, d)$. In this important special case $T_\lambda$ is a discrete analogue of the Riesz potential $I_\alpha=(-\Delta)^{-\frac{\alpha}{2}}$, and 
$M_\lambda$ is the dyadic fractional maximal operator.  

\end{abstract}
\maketitle
\tableofcontents

\section*{Notation}
\begin{tabular}{c l}
$\cd$ & The collection of all the dyadic cubes $Q$ in $\br^d$.\\
$L^p(\mu)$ & The Lebesgue space with respect to a measure $\mu$, equipped with the norm \\
& $\norm{f}_{L^p(\mu)}:=(\int \abs{f}^p \dmu)^{\frac{1}{p}}.$\\
$f^{p,q}(\mu)$ & The discrete Littlewood--Paley space, equipped with the norm\\
&  $\norm{a}_{f^{p,q}(\mu)}:=\big(\int \big( \sum_Q \abs{a_Q}^q 1_Q\big)^{\frac{p}{q}} \dmu\big)^{\frac{1}{p}},  \text{ when } 0<p<\infty, \, 0<q\le \infty,$\\
&  $\norm{a}_{f^{\infty, q}(\mu)}:=  \Big( \sup_Q  \frac{1}{\mu(Q)}\sum_{R\subseteq Q} \abs{a_R}^q \, \mu(R) \Big)^{\frac{1}{q}},\text{when } p=\infty, \, 0<q<\infty$.\\
$\angles{f}^\mu_Q$ & The average of the function $|f|$ on a cube $Q$,\\
& $\angles{f}^\mu_Q:=\frac{1}{\mu(Q)} \int_Q |f| \dmu.$\\
$p'$& The H\"older conjugate $p'\in[1,\infty]$ of an exponent $p\in[1,\infty]$,\\
& $p':=\frac{p}{p-1}$.\\
$T_\lambda(\cdotroomy\sigma)$ & The dyadic summation operator, \\ 
&$T_\lambda(f\sigma):=\sum_{Q \in \cq} \lambda_Q \, \angles{f}^\sigma_Q \, 1_Q.$\\
$M_\lambda(\cdotroomy\sigma)$ & The dyadic maximal operator, \\ 
&$M_\lambda(f\sigma):=\sup_{Q \in \cq} \lambda_Q \, \angles{f}^\sigma_Q \, 1_Q.$\\
$\rho^{{\rm sum}}_Q$& The localized sum of the $T_\lambda(\cdotroomy \sigma)$  coefficients,  \\
&$\rho_Q^{{\rm sum}}:= \sum_{R\subseteq Q}\lambda_R 1_R.$\\
$\rho^{\sup}_Q$& The localized supremum of the  $M_\lambda(\cdotroomy \sigma)$ coefficients,  \\
&$\rho_Q^{\sup}:= \sup_{R\subseteq Q}\lambda_R 1_R.$\\
$\Lambda^{{\rm sum}}_{\gamma,Q}$ & The $\gamma$-average of $\rho_Q^{{\rm sum}}$, \\
&$\Lambda^{{\rm sum}}_{\gamma,Q}:=\Big(\frac{1}{\omega(Q)}\int_Q (\rho^{{\rm sum}}_Q)^\gamma\domega \Big)^{\frac{1}{\gamma}}, \, \gamma \in \br\setminus\{0\}.
$\\
$\Lambda_{Q}$ & The average of $\rho_Q^{{\rm sum}}$, \\
&$\Lambda_{Q}:=\Lambda_{Q}^{{\rm sum}}=
\frac{1}{\omega(Q)}\sum_{R\subseteq Q} \lambda_R \omega(R).
$\\
$\Lambda^{\sup}_{\gamma,Q}$ & The $\gamma$-average of $\rho_Q^{\sup}$, \\
&$\Lambda^{\sup}_{\gamma,Q}:=\Big(\frac{1}{\omega(Q)}\int_Q (\rho^{\sup}_Q)^\gamma\domega \Big)^{\frac{1}{\gamma}}, \, \gamma \in \br\setminus\{0\}.
$\\
$a^{-1}$& For a family $a:=\{a_Q\}$, the family $a^{-1}$ is defined by $a^{-1}:=\{a_Q^{-1}\}$. 
\end{tabular}
\smallskip

The least constant in the $L^p(\sigma)\to L^q(\omega)$ two-weight norm inequality 
\eqref{eq:norminequality}  
for the operator $T=T_\lambda(\cdotroomy \sigma)$ is denoted by $\norm{T_\lambda(\cdotroomy \sigma)}_{L^p(\sigma)\to L^q(\omega)}$, and the least constant for the operator $T=M_\lambda(\cdotroomy \sigma)$ by $\norm{M_\lambda(\cdotroomy \sigma)}_{L^p(\sigma)\to L^q(\omega)}$.

The uppercase letters  $P,Q,R,S$ are reserved for dyadic cubes. The  indexing \quotes{$Q\in \cd$} is abbreviated as \quotes{$Q$} in the  indexing of summations, and omitted in the  indexing of families (and similarly for the cubes $P,R,S$). 

The lowercase letters  $a,b, \ldots$ are reserved for various families $a:=\{a_Q\}$, $b:=\{b_Q\}, \ldots$ of non-negative reals, and $\lambda:=\{\lambda_Q\}$ for the fixed family of non-negative reals associated with the operators $T_\lambda(\cdotroomy \sigma)$ and $M_\lambda(\cdotroomy \sigma)$.

We follow the usual convention $\frac{0}{0}:=0$. 

The standing assumption is that $p\in(1,\infty)$, $q\in(0,\infty)$ and $q<p$. Hence only further restrictions on the exponents are mentioned.

\section{Introduction} 

Let $\sigma$ and $\omega$ be locally finite Borel measures on $\br^d$, and 
let $\lambda=\{\lambda_Q\}_{Q \in \cd}$ 
be a sequence of non-negative reals indexed by the dyadic cubes $Q \in \cd$. We study the two-weight inequalities 
\begin{equation} \label{eq:norminequality} 
\lVert T(f\sigma) \rVert_{L^q(\omega)}\leq C \lVert f \rVert_{L^p(\sigma)}, \quad \text{for all} \, \,  f \in L^p(\sigma),
\end{equation} 
in the range of the exponents $0<q<p$ and $p>1$. It is a standing assumption  throughout this article that the exponents are in this range and hence only further restrictions on the exponents are mentioned. (Here, $T$ is either the dyadic summation operator $T_\lambda(\cdotroomy \sigma)$, or the dyadic maximal operator $M_\lambda(\cdotroomy \sigma)$, both of which are defined in Notation.)

This range of exponents
appeared recently in applications to nonlinear elliptic PDE \cite{cao2017}, \cite{quinn2017}, 
\cite{verbitsky2017}; in this case $\lambda_Q = \sigma(Q) \, |Q|^{\frac{\alpha}{d}-1}$ with $\alpha \in (0, d)$ and so
$T_\lambda(f \sigma)= I^{\cd}_\alpha(f  \sigma)$ is the dyadic Riesz potential, a discrete 
analogue of the classical Riesz potential $I_\alpha=(-\Delta)^{-\frac{\alpha}{2}}$, and $M_\lambda(f \sigma)= M^{\cd}_\alpha(f  \sigma)$ is the dyadic fractional maximal operator (see \cite{adams1996}, \cite{cascante2006}, \cite{sawyer1992}, \cite{verbitsky1992}). 

Nevertheless, this range of exponents is still insufficiently understood, especially 
in the range $0<q<1<p$ for the summation operator. (The characterization in the 
case $1<q<p$ was completed recently by Tanaka \cite[Theorem 1.3]{tanaka2014}.) 

For the two-weight inequality \eqref{eq:norminequality} in its full generality, the known sufficient and necessary conditions are complicated. The conditions that characterize the inequality for maximal operators in this range are, in essence, conditions that are required to hold uniformly over all linearizations of maximal operators (see  \cite[Theorem 2]{verbitsky1992} by Verbitsky, \cite[Theorem 7.8]{hanninen2015} by H\"anninen, and \cite[Theorem 5.2]{hanninen2016}) by H\"anninen, Hyt\"onen, and Li). These conditions are described in Subsection \ref{subsection:statement_scale}. Similarly, the conditions that characterize  the inequality for summation operators in the subrange $0<q<1$ are required to hold over all possible factorizations (see \cite[Theorem 1.1. and Theorem 1.2]{hanninen2017a} by the authors).  Although these characterizations provide us with alternative viewpoints at these inequalities and offer an alternative starting point for their study, such conditions  are  difficult to verify in applications. 
 
To ameliorate this problem in the case of summation operators, we introduced earlier  a scale of conditions that depends on an integrability parameter and showed that, on this scale, the sufficiency and necessity conditions are separated by a certain  integrability gap (see \cite[Theorem 1.3]{hanninen2017a} by the authors). In this article, we now introduce an analogous scale of conditions for maximal operators and show that, on this scale, the sufficiency and necessity conditions are separated only by an arbitrarily small integrability gap (see Proposition \ref{proposition:sufficient_and_necessary_maximal} for the precise statement).

Under the additional assumption that the measures $\sigma$ and $\omega$ satisfy the $A_\infty$ condition with respect to each other, simple conditions for both summation and maximal operators are known in many ranges of exponents $p$ and $q$. In this article, we complete this picture by addressing the remaining case: the case of summation operators and the range $p\in(1,\infty)$, $q\in(0,\infty)$, and $q<p$ (see Proposition \ref{proposition:ainfty} for the precise statement). The proof is based on a characterization of multipliers of Carleson coefficients (see Proposition \ref{proposition:multipliers_carleson} for the precise statement).

Although the summation operator and supremum operator can both be viewed on the scale of vector-valued operators
$$
T_{r}(f\sigma):=\Big( \sum_{Q\in\cd} (\lambda_Q \angles{f}_Q^\sigma 1_Q)^r 
\Big)^{\frac{1}{r}}\quad r\in(0,\infty],
$$
 the characterizations of them, both the statements and the proofs, seem to be very different from each other and, to the best of the authors' knowledge, no explicit connections between the inequalities for summation and maximal operators are known. In this article, we find that, in the range $q\in (1,\infty)$, the inequality for summation operators can be characterized in terms of inequalities for related maximal operators (see Proposition \ref{proposition:connection} for the precise statement). This maximal-type condition can also be regarded as an alternative to the known potential-type condition (see  \cite[Theorem A]{cascante2006} by Cascante, Ortega, and Verbitsky, and \cite[Theorem 1.3]{tanaka2014} by Tanaka). The known potential-type condition is described in Subsection \ref{subsection:statement_connection}.

Next, we present in more detail each of our results and how they are related to the earlier results in the literature. 

\section{Statements of results}

\subsection{Scale of conditions for maximal operators}\label{subsection:statement_scale}Let $0<q<p<\infty$ and $p>1$. We study the two-weight norm inequality
\begin{equation}\label{max} 
\lVert M_\lambda(f\sigma) \rVert_{L^q(\omega)}\leq C \lVert f \rVert_{L^p(\sigma)}, \quad \text{for all} \, \,  f \in L^p(\sigma).
\end{equation} 
In the general case, the known sufficient and necessary conditions  are complicated and difficult to apply, whereas only in the limited particular cases simpler and more easily applicable conditions are known. 
For general measures $\sigma$ and $\omega$ and coefficients $\lambda$, the following complicated conditions are known:
\begin{itemize}
\item For every collection $\cq$ of dyadic cubes, we define the auxiliary function $\lambda_{\cq}$  by $$\lambda_{\cq}(x):=\inf_{Q\in \cq} \sup_{\substack{R\in\cq:\\ R\subseteq Q}} \lambda_R \, 1_R(x).$$
Inequality \eqref{max} holds if and only if there exists a constant $C>0$ such that
\begin{equation}\label{max-char} 
\int \sup_{x\in Q \in \cq} \left( \frac{\int_Q \lambda^q_{\cq}(y) \domega(y)}{\sigma(Q)}\right)^{\frac{q}{p-q}} \lambda_{\cq}^q(x)\domega(x)\leq C
\end{equation}
for all collections $\cq$ of dyadic cubes. This characterization was obtained by Verbitsky \cite[Theorem 2]{verbitsky1992}.

\item Inequality  \eqref{max} holds if and only if there exists a constant $C>0$ such that
\begin{equation}
\label{temp41}
\int \left( \sum_{Q\in\cd} \lambda_Q^q \frac{\omega(E_Q)}{\sigma(Q)}1_Q \right)^{\frac{p-q}{q}} \dsigma \leq C
\end{equation}
for all collections $\{E_Q\}$ of pairwise disjoint sets $E_Q$ such that $E_Q\subseteq Q$. This characterization was observed by H\"anninen \cite[Theorem 7.8]{hanninen2015}, and a variant of it by H\"anninen, Hyt\"onen, and Li \cite[Theorem 5.2]{hanninen2016}.
\end{itemize}
For particular measures $\sigma$ and $\omega$, or for particular coefficients $\lambda$, from these conditions the following simpler conditions follow:
\begin{itemize}
\item Assume that the coefficients $\lambda$ satisfy 
$$
\sup_{R: R\subseteq Q} \lambda_R 1_R \eqsim \lambda_Q
$$
for all dyadic cubes $Q$. This is an analogue of the so called {\it dyadic logarithmic bounded oscillation condition} (DLBO) for summation operators (see, for example, \cite{cascante2004}). 
 Then inequality \eqref{max} holds if and only if there exists a constant $C>0$ such that
$$
\int \left( \sup_{Q \in \cd}  \lambda_Q \big(\frac{\omega(Q)}{\sigma(Q)}\big)^{\frac{1}{p}} 1_Q \right)^{\frac{pq}{p-q}} \domega\leq C.
$$
\item Assume that the measures $\sigma$ and $\omega$ satisfy the $A_\infty$ condition with respect to each other and have no point masses. Then, by Corollary \ref{corollary:a_infty_disjoint}, for each collection $\{E_Q\}$ of disjoint sets with $E_Q\subseteq Q$ there exists a collection $\{F_Q\}$ of disjoint sets with $F_Q\subseteq Q$ such that 
$$
\frac{\sigma(E_Q)}{\sigma(Q)}\leq [\omega]_{A_\infty(\sigma)} \frac{\omega(F_Q)}{\omega(Q)}
$$
for all dyadic cubes $Q$, and conversely. From combining this with the condition \eqref{temp41} it follows that the two-weight norm inequality holds if and only if there exists a constant $C>0$ such that
$$
\int \left( \sup_{Q \in \cd}  \lambda_Q \big(\frac{\omega(Q)}{\sigma(Q)}\big)^{\frac{1}{q}} 1_Q \right)^\frac{pq}{p-q} \dsigma\leq C.$$
\end{itemize}

In this paper, we introduce a scale of simple conditions that depend on an integrability parameter, and prove that the necessity and sufficiency on this scale are separated only by an arbitrarily small integrability gap. For each integrability parameter $\gamma\in(-\infty,\infty)$, we define the localized auxiliary quantity $\Lambda^{\sup}_{\gamma,Q}$ by 
$$\Lambda^{\sup}_{\gamma,Q}:= \left(\frac{1}{\omega(Q)}\int_Q ( \sup_{R: R\subseteq Q} \lambda_R 1_R )^\gamma \domega \right)^{\frac{1}{\gamma}}.$$ 
Our result reads as follows:
\begin{proposition}[Scale of conditions 
for maximal operators]\label{proposition:sufficient_and_necessary_maximal}Let $p\in(1,\infty)$, $q\in(0,\infty)$, and $q<p$. The following assertions hold:
\begin{enumerate}[label=(\roman*)]
\item (Sufficient condition) We have
\begin{equation}
\label{temp40}
\norm{M_\lambda(\cdotroomy)}_{L^p(\sigma)\to L^q(\omega)}\lesssim_{p,q} \left(\int \sup_Q  \Big(\frac{\omega(Q)}{\sigma(Q)} \Big)^{\frac{q}{p-q}} \lambda_Q^q (\Lambda_{q,Q}^{\sup})^{\frac{q^2}{p-q}} 1_Q \domega\right)^{\frac{p-q}{pq}}.
\end{equation}
\item (Necessary condition) Let $\epsilon>0$ be an arbitrarily small positive real. We have
$$
 \left(\int \sup_Q \Big(\frac{\omega(Q)}{\sigma(Q)} \Big)^{\frac{q}{p-q}} (\Lambda^{\sup}_{(q-\epsilon),Q})^{\frac{pq}{p-q}} 1_Q \domega\right)^{\frac{p-q}{pq}} \lesssim_{\epsilon,p,q} \norm{M_\lambda(\cdotroomy)}_{L^p(\sigma)\to L^q(\omega)}.
$$
\end{enumerate}
\end{proposition}
\begin{remark}Our condition \eqref{temp40} is sufficient in the general case. In addition, it is also necessary in the particular case where $\sup_{R: R\subseteq Q} \lambda_R 1_R \eqsim \lambda_Q$, and also in the particular case where $\sigma$ and $\omega$ are $A_\infty$ measures with respect to each other. Thus, our condition includes the above-listed earlier particular cases in which simple conditions were known. Furthermore, our sufficient condition is close to being necessary even in the general case, since the sufficient condition \eqref{temp40}  becomes necessary once the integrability parameter $q$ in the quantity $\Lambda^{\sup}_{q,Q}$ is lowered by an arbitrarily small  $\epsilon>0$.\end{remark}

\subsection{Characterization for summation operators under the $A_\infty$ assumption} In the case where the  measure $\sigma$ satisfies the $A_\infty$ condition with respect to  $\omega$, the two-weight norm inequality 
\begin{equation}\label{sum} 
\lVert T_\lambda(f\sigma) \rVert_{L^q(\omega)}\leq C \lVert f \rVert_{L^p(\sigma)}, \quad \text{for all} \, \,  f \in L^p(\sigma),
\end{equation} 
can be characterized by simple integral conditions. In this work, we use the Fujii--Wilson $A_\infty$ condition. Since the Coifman--Fefferman $A_\infty$ condition is also used in related earlier work, such as \cite{verbitsky1992}, we recall both of these conditions and their relations. The conditions are as follows:
\begin{enumerate}
\item\label{item:fw} (Fujii--Wilson) A measure $\sigma$ is said to satisfy the {\it dyadic Fujii--Wilson $A_\infty$ condition} with respect to a measure $\omega$ if there exists a constant $C$ such that, for every dyadic cube $Q$, we have $$\int \sup_{R\in\cd : R\subseteq Q} \Big(\frac{\sigma(R)}{\omega(R)} 1_R\Big) \, d \omega\leq C \sigma(Q).$$
The least such constant $C$ is called the {\it Fujii--Wilson $A_\infty$ characteristic} and denoted by $[\sigma]_{A_\infty(\omega)}$.
\item \label{item:cf} (Coifman--Fefferman) A measure $\sigma$ is said to satisfy the {\it dyadic Coifman--Fefferman $A_\infty$ condition} with respect to a measure $\omega$ if there exist $\alpha, \beta \in(0,1)$ such that for every dyadic cube and every subset $E\subseteq Q$ we have that $\omega(E)\leq \beta \omega(Q)$ implies $\sigma(E)\leq \alpha \sigma(Q)$. 
\end{enumerate}
We observe, by contraposition and by taking complement, that the Coifman--Fefferman condition is symmetric in the measures $\sigma$ and $\omega$. Some relations between the conditions are as follows:
\begin{itemize}
\item For non-doubling measures, the Coifman--Fefferman condition is in general strictly stronger than the Fujii--Wilson condition. For a proof that \eqref{item:cf} implies \eqref{item:fw}, see, for example, \cite[Proof of Lemma 2.5]{hanninen2017}. To see that measures may satisfy the Fujii--Wilson condition, but fail to satisfy the Coifman--Fefferman condition, notice that, by the Lebesgue differentiation theorem, the Coifman--Fefferman condition requires that $\sigma$ is absolutely continuous with respect to  $\omega$, whereas the Fujii--Wilson condition does not require this. Accordingly, the case with $\sigma$ being Lebesgue measure and $\omega$ a Dirac measure is an example of measures satisfying \eqref{item:fw} but not  \eqref{item:cf}. 
\item Nevertheless, the conditions are equivalent 
provided both $\omega$ and $\sigma$ are doubling \cite[Theorem 1]{fujii1978}. Moreover, the doubling properties of the measures were originally assumed in the Coifman--Fefferman condition \cite{coifman1974}. Furthermore, because the Coifman--Fefferman condition is symmetric in the measures, in the case of doubling measures, $\sigma$ satisfies the 
Fujii--Wilson $A_\infty$ condition with respect to $\omega$ if and only if  $\omega$ satisfies the same condition with respect to $\sigma$ .
\end{itemize}

Under the $A_\infty$ assumption, the following simpler (than in the general case) characterizations are known:
\begin{itemize}
\item \textit{The subrange $1<p\leq q<\infty$ for maximal and summation operators.}  H\"anninen \cite[Theorem 1.5]{hanninen2017} noticed that the two-weight norm inequality for the summation operators is characterized by testing the bilinear estimate against the indicator functions of cubes:
\begin{equation*}
\begin{split}
&\sup_Q \frac{\int 1_QT_\lambda(1_Q \sigma) \domega}{\omega(Q)^{\frac{1}{q'}}  \sigma(Q)^{\frac{1}{p}}} \\
&\lesssim \norm{T_\lambda(\cdotroomy \sigma)}_{L^p(\sigma)\to L^q(\omega)}\\
&\lesssim_{p,q} ([\sigma]_{A_\infty(\omega)}^{\frac{1}{p}}+[\omega]_{A_\infty(\sigma)}^{\frac{1}{q'}}) \sup_Q \frac{ \int 1_QT_\lambda(1_Q \sigma) \domega }{\omega(Q)^{\frac{1}{q'}} \sigma(Q)^{\frac{1}{p}}}.
\end{split}
\end{equation*}
A similar characterization 
holds for the maximal operators as well, and it can be proven, for example,  by a parallel  stopping cubes argument analogous to the argument appearing in \cite{hanninen2017}.
\item  \textit{The subrange $0<q<p$ and $p>1$ for maximal operators.} Verbitsky \cite{verbitsky1992} proved that the two-weight norm inequality for the maximal operators is characterized by a simple integral condition:
\begin{equation*}
\begin{split}
& [\omega]_{A_\infty(\sigma)}^{-\frac{1}{q}} \left( \int \big(\sup_Q \lambda_Q \big( \frac{\omega(Q)}{\sigma(Q)}\big)^{\frac{1}{q}} 1_Q \big)^{\frac{pq}{p-q}} \dsigma \right)^{\frac{p-q}{pq}} \\
&\lesssim \norm{M_\lambda(\cdotroomy \sigma)}_{L^p(\sigma)\to L^q(\omega)}\\
&\lesssim_{p,q} [\sigma]_{A_\infty(\omega)}^{-\frac{1}{q}} \left( \int \big(\sup_Q \lambda_Q \big( \frac{\omega(Q)}{\sigma(Q)}\big)^{\frac{1}{q}} 1_Q \big)^{\frac{pq}{p-q}} \dsigma \right)^{\frac{p-q}{pq}}.
\end{split}
\end{equation*}
\end{itemize}

In this paper, we address the remaining case:  \textit{The subrange  $0<q<p$ and $p>1$ for summation operators.} 
For brevity, we write
\begin{equation*}
I_{\sigma,\omega,p,q,\lambda}:= \left( \int \big(\sum_Q \lambda_Q \Big(\frac{\omega(Q)}{\sigma(Q)}\Big)^{\frac{1}{q}} 1_Q \big)^{\frac{pq}{p-q}} \dsigma \right)^\frac{p-q}{pq}
\end{equation*}
for the integral expression, whose finiteness is sufficient and necessary for  inequality \eqref{sum}:

\begin{proposition}[Characterization under the $A_\infty$ assumption]\label{proposition:ainfty}Let $\sigma$ and $\omega$ be measures that satisfy the $A_\infty$ condition with respect to each other. Let $p\in(1,\infty)$ and $q\in(0,\infty)$ be such that $q<p$. Then we have the following characterization by subranges:
\begin{itemize}
\item In the subrange $q\in(0,1]$, we have
\begin{equation*}
\begin{split}
&[\omega]_{A_\infty(\sigma)}^{-\frac{1-q}{q}} I_{\sigma,\omega,p,q,\lambda}\lesssim_{p,q} \norm{T_\lambda (\cdotroomy \sigma)}_{L^p(\sigma)\to L^q(\omega)}  \lesssim_{p,q} [\sigma]_{A_\infty(\omega)}^{\frac{1-q}{q}} I_{\sigma,\omega,p,q,\lambda}.
\end{split}
\end{equation*}

\item In the subrange $q\in(1,\infty)$, we have
\begin{equation*}
\begin{split}
&[\sigma]_{A_\infty(\omega)}^{-\frac{q-1}{q}} I_{\sigma,\omega,p,q,\lambda}
\lesssim_{p,q} \norm{T_\lambda (\cdotroomy \sigma)}_{L^p(\sigma)\to L^q(\omega)} 
 \lesssim_{p,q} [\omega]_{A_\infty(\sigma)}^{\frac{q-1}{q}} I_{\sigma,\omega,p,q,\lambda}.
\end{split}
\end{equation*}

\end{itemize}
\end{proposition}
\begin{remark} In the subrange $q\in(1,\infty)$, by the $L^q(\omega)-L^{q'}(\omega)$ duality, we have
$$
\norm{T_{\{\lambda_Q\}}(\cdotroomy \sigma)}_{L^p(\sigma)\to L^q(\omega)}=\norm{T_{\{\lambda_Q \frac{\omega(Q)}{\sigma(Q)}\}}(\cdotroomy \omega)}_{L^{q'}(\omega)\to L^{p'}(\sigma)}.
$$
Therefore, by Proposition \ref{proposition:ainfty}, we also have 
\begin{equation}
\begin{split}
&[\omega]_{A_\infty(\sigma)}^{-\frac{1}{p}} I^*_{\sigma,\omega,p,q,\lambda} \lesssim_{p,q} \norm{T_\lambda (\cdotroomy \sigma)}_{L^p(\sigma)\to L^q(\omega)}  \lesssim_{p,q}  [\sigma]_{A_\infty(\omega)}^{\frac{1}{p}} I^*_{\sigma,\omega,p,q,\lambda},
\end{split}
\end{equation}
where the dual integral expression $I^*_{\sigma,\omega,p,q,\lambda}$ is defined by
$$
I^*_{\sigma,\omega,p,q,\lambda}:= \left( \int \big(\sum_Q \lambda_Q \Big(\frac{\omega(Q)}{\sigma(Q)}\Big)^{\frac{1}{p}} 1_Q \big)^{\frac{pq}{p-q}} \domega \right)^{\frac{p-q}{pq}},
$$
and is related to the expression $I_{\sigma,\omega,p,q,\lambda}$ via interchanging $\lambda_Q \frac{\omega(Q)}{\sigma(Q)}$ and $\lambda_Q$, $q'$ and $p$, and $\omega$ and $\sigma$.
\end{remark}

\subsection{Inequality for summation operators via maximal operators}\label{subsection:statement_connection} In this section, we show that the two-weight norm inequality \eqref{sum} for the summation operator is equivalent to a pair of two-weight norm  inequalities for certain related maximal operators:
\begin{proposition}\label{proposition:connection}Let $1<q<p<\infty$. Let $\{\lambda_Q\}$ be non-negative reals. Then the following assertions are equivalent:
\begin{enumerate}[label=(\roman*)]
\item\label{assertion:inequality_summation} Inequality \eqref{sum}  holds, that is, 
\begin{equation}\label{temp50}
\norm{\sum_Q \lambda_Q \angles{f}^\sigma_Q 1_Q}_{L^q(\omega)}\lesssim_{p,q}   \norm{f}_{L^p(\sigma)}\quad \text{for all functions $f$.}
\end{equation}
\item\label{assertion:inequalities_maximal} The following two-weight norm  inequalities hold for the related maximal operators:
$$
\left\{
\begin{aligned}
&\norm{\sup_Q \Lambda_Q \angles{f}^\sigma_Q 1_Q}_{L^q(\omega)}\lesssim_{p,q}   \norm{f}_{L^p(\sigma)} \quad \text{for all functions $f$} ,  \\
& \norm{\sup_Q \frac{\omega(Q)}{\sigma(Q)}\Lambda_Q \angles{g}^\omega_Q 1_Q}_{L^{p'}(\sigma)}\lesssim_{p,q}   \norm{g}_{L^{q'}(\omega)} \quad \text{for all functions $g$} , 
  \end{aligned}
\right.
$$
where $\Lambda_Q= \Lambda^{{\rm sum}}_{Q}:= \frac{1}{\omega(Q)} \sum_{R\subseteq Q} \lambda_R \omega(R)$. 
\end{enumerate}
\end{proposition}
In this range $1<q<p$, inequality \eqref{temp50}  for the summation operator can also be characterized by the following two potential-type conditions:
\begin{equation}\label{eq:wolff_potential_condition}
\int \big(W^{p'}_{\lambda,\sigma}[\omega]\big)^{\frac{(p-1)q}{p-q}} \domega < \infty\quad \text{ and } \quad \int \big(W^{q}_{\lambda,\omega}[\sigma]\big)^{\frac{(q'-1)p'}{q'-p'}} \dsigma < \infty .
\end{equation}
The necessity  of \eqref{eq:wolff_potential_condition} for \eqref{temp50} follows from the results of Cascante, Ortega, and Verbitsky
 \cite[Theorem 2.1]{cascante2006}, and the sufficiency was established later by Tanaka \cite[Theorem 1.3]{tanaka2014}. 
Here, the {\it discrete Wolff potential} $W^{p'}_{\lambda,\sigma}[\omega]$ associated with the summation operator $T_{\lambda}(\cdotroomy \sigma):L^p(\sigma)\to L^q(\omega)$ is defined  by
\begin{equation*}\label{eq:def:wolfpotential}
W^{p'}_{\lambda,\sigma}[\omega]:= \sum_Q 1_Q \lambda_Q \Big(\frac{\omega(Q)}{\sigma(Q)}\Big)^{p'-1} (\Lambda_{Q})^{p'-1},
\end{equation*}
and the dual Wolff potential $W^{q}_{\lambda,\omega}[\sigma]$ is the discrete Wolff potential associated with the adjoint operator $T_{\{\lambda_Q \frac{\omega(Q)}{\sigma(Q)}\}}(\cdotroomy \omega):L^{q'}(\omega)\to L^{p'}(\sigma)$. (Hence, $W^{q}_{\lambda,\omega}[\sigma]$ has an expression similar to $W^{p'}_{\lambda,\sigma}[\omega]$, but with $\lambda_Q$ replaced by $\lambda_Q \frac{\sigma(Q)}{\omega(Q)}$, $p$ by $q'$, $q$ by $p'$, respectively, and with $\sigma$ and $\omega$ swapped.)

Whereas in the range $1<q<p$ the potential-type condition 
\eqref{eq:wolff_potential_condition} is both sufficient and necessary, in the more difficult range $0<q<1$ and $p>1$ no explicit necessary and sufficient condition is known. The authors hope that the connection between the  two-weight norm  inequality for the summation operator and the two-weight inequalities for the related maximal operators (Proposition \ref{proposition:connection}) may be extended from the range $1<q<\infty$  to the range $0<q<1$, which would be useful in finding a concrete necessary and sufficient condition for summation operators.

\section{Preliminaries}
\subsection{Discrete Littlewood--Paley spaces}We recall the definition of discrete Littlewood--Paley spaces $f^{r,s}(\mu)$ for exponents $p\in (0, +\infty]$, $q\in\br\setminus\{0\}$, and a locally finite Borel measure $\mu$ on $\br^d$. Essentially this scale of spaces was introduced by Frazier and Jawerth \cite{frazier1990} in the case of Lebesgue measure  (see \cite{cohn2000} in the general case). The {\it discrete Littlewood--Paley norm} $\norm{a}_{f^{p,q}(\mu)}$ of a family $\{a_Q\}_{Q\in\cd}$ of nonnegative reals is defined by cases as follows:
\begin{itemize}
\item For $p\in(0,\infty)$ and $q\in\br\setminus\{0\}$,
$$
\norm{a}_{f^{p,q}(\mu)}:= \Big(\int \big(  \sum_Q a_Q^q 1_Q\big)^{\frac{p}{q}} \dmu\Big)^{\frac{1}{p}}.
$$
\item For $p\in(0,\infty)$ and $q=\infty$,
$$
\norm{a}_{f^{p,\infty}(\mu)}:= \Big(\int \big( \sup_Q a_Q 1_Q \big)^p \dmu\Big)^{\frac{1}{p}}.
$$
\item For $p=\infty$ and $q\in\br\setminus\{0\}$,
$$
\norm{a}_{f^{\infty,q}(\mu)}:= \sup_{Q }\Big( \frac{1}{\mu(Q)} \sum_{R\subseteq Q} a_R^q \mu(R) \Big)^{\frac{1}{q}}.
$$
\item For $p=\infty$ and $q=\infty$,
$$
\norm{a}_{f^{\infty,\infty}(\mu)}:=\sup_Q a_Q.
$$
\end{itemize}
The discrete Littlewood--Paley norm can be computed via duality as follows \cite[Theorem 4 and Remark 5]{verbitsky1996}:
\begin{proposition}[Computing norm by duality in discrete Littlewood--Paley spaces]\label{proposition:normduality} Let $p,q\in[1,\infty]$. Let $\mu$ be a locally finite Borel measure.  Then, we have
$$
\norm{a}_{f^{p,q}(\mu)}\eqsim_{p,q} \sup_{\norm{b}_{f^{p',q'}(\mu)}\leq 1} \sum_Q a_Q b_Q \mu(Q)
$$
for every family $\{a_Q\}_{Q\in\cd}$.
\end{proposition}
\begin{remark}In particular, in the case $p=\infty,q=1$, the dual norm formula reads that the dual estimate
$$
\sum_{Q\in\cd} a_Q b_Q \mu(Q) \leq C \norm{\sup_{Q\in\cd} b_Q 1_Q}_{L^{1}(\mu)}\quad \text{for all families $b$}
$$
holds if and only if the Carleson condition $$\sup_{Q\in\cd} \frac{1}{\mu(Q)}\sum_{R\in\cd: R\subseteq Q} a_R \mu(R) \leq C $$ holds, which is a dyadic form 
of the Carleson imbedding theorem. \end{remark}

In the Littlewood--Paley spaces the following factorization holds (\cite[Theorem 2.4]{cohn2000}):
\begin{proposition}[Factorization in discrete Littlewood--Paley spaces]
\label{proposition:factorization}Let $\mu$ be a locally finite Borel measure on $\br^d$. Let $p,p_1,p_2\in(0,\infty]$ and $q,q_1,q_2\in(0,\infty]$ be exponents that satisfy the H\"older relations:
$$
\frac{1}{p}=\frac{1}{p_1}+\frac{1}{p_2}\quad \text{ and } \quad\frac{1}{q}=\frac{1}{q_1}+\frac{1}{q_2}. 
$$
Then, the following assertions hold:
\begin{enumerate}[label=(\roman*)]
\item Every $a\in f^{p_1,q_1}$ and $b\in f^{p_2,q_2}$ satisfy the estimate
$$
\norm{ab}_{f^{p,q}(\mu)}\lesssim_{q,p}\norm{a}_{f^{p_1,q_1}(\mu)}\norm{b}_{f^{p_2,q_2}(\mu)}.
$$
\item For each $c\in f^{p,q}(\mu)$ there exist $a\in f^{p_1,q_1}(\mu)$ and $b\in f^{p_2,q_2}(\mu)$ such that $c=ab$ and
$$
\norm{a}_{f^{p_1,q_1}(\mu)} \, \norm{b}_{f^{p_2,q_2}(\mu)} \lesssim_{p,q} \norm{c}_{f^{p,q}(\mu)}.
$$
\end{enumerate}
\end{proposition}

\subsection{Dyadic Hardy--Littlewood maximal inequality}We recall the dyadic Hardy--Littlewood maximal inequality. The {\it dyadic Hardy--Littlewood maximal operator} $M^\mu(\cdotroomy)$ is defined by 
$$
M^\mu(f):=\sup_{Q\in\cd} \angles{f}^\mu_Q 1_Q.
$$

\begin{lemma}[Dyadic Hardy--Littlewood maximal inequality] Let $p\in(1,\infty]$, and let $\mu$ be a locally finite Borel measure  on $\br^d$. Then
$$
\norm{M^\mu(f)}_{L^p(\mu)}\lesssim_p \norm{f}_{L^p(\mu)}
$$
for every $f\in L^p(\mu)$.
\end{lemma}

\subsection{Equivalent expressions}
\begin{lemma}[Equivalent expressions; Proposition 2.2 in \cite{cascante2004}]\label{lemma:eq_exp}Let $p\in(1,\infty)$. Then the following expressions are comparable:
\begin{equation}
\begin{split}
&\int \big( \sum_Q a_Q 1_Q \big)^p \dmu\\
&\eqsim_p \sum_Q a_Q \mu(Q) \big(\frac{1}{\mu(Q)} \sum_{R\subseteq Q} a_R \mu(R) \big)^{p-1}\\
&\eqsim_p \int \big( \sup_{Q} \frac{1_Q}{\mu(Q)} \sum_{R\subseteq Q} a_R \mu(R) \big)^p \dmu.
\end{split}
\end{equation}
\end{lemma}
\subsection{Reformulations of the two-weight norm inequalities}We reformulate the two-weight norm inequalities in terms of coefficients in place of functions. These reformulations are used in Subsection \ref{subsection:proof_connection} to pass between the two-weight norm inequality for summation operators and the related inequalities for related maximal operators.

\begin{lemma}[Reformulations for summation operators]\label{lemma:reformulations_summation}Let  $p,q\in(1,\infty)$. Then the following estimates are equivalent:
\begin{enumerate}[label=(\roman*)]
\item \label{temp:a1} We have
$$
\norm{\sum_P \lambda_P \angles{f}^\sigma_P 1_P }_{L^q(\omega)}\lesssim_{p,q} C \norm{f}_{L^p(\sigma)}
$$
for all functions $f$.
\item  \label{temp:a2}We have
$$
\sum_P \lambda_P \omega(P) \angles{f}^\sigma_P \angles{g}^\omega_P \lesssim_{p,q} C \norm{f}_{L^p(\sigma)} \norm{g}_{L^{q'}(\omega)}
$$
for all functions $f$ and $g$.
\item \label{temp:a3} We have
$$
\sum_P \lambda_P \omega(S) a_P b_P \lesssim_{p,q} C \norm{\sup_Q a_Q 1_Q}_{L^p(\sigma)} \norm{\sup_R b_R 1_R}_{L^{q'}(\omega)}
$$
for all families $a$ and $b$.
\item  \label{temp:a4} We have
$$
\sum_P \lambda_P \omega(P) \big(\sum_{Q\supseteq P} \tilde{a}_Q\big) \big( \sum_{R\supseteq P} \tilde{b}_R \big) \lesssim_{p,q} C \norm{\sum_Q \tilde{a}_Q 1_Q}_{L^p(\sigma)} \norm{\sum_R \tilde{b}_R 1_R}_{L^{q'}(\omega)}
$$
for all families $\tilde{a}$ and $\tilde{b}$.
\end{enumerate}
\end{lemma}
\begin{proof} The equivalence between  estimates \ref{temp:a1} and \ref{temp:a2} follows from the $L^q(\omega)-L^{q'}(\omega)$ duality.

Estimate \ref{temp:a2} implies estimate \ref{temp:a3} via the substitutions $f:=\sup_Q a_Q 1_Q$ and $g:=\sup_R b_R 1_R$, and, conversely,   estimate \ref{temp:a3} implies estimate \ref{temp:a2}  via the substitutions $a_Q:= \angles{f}^\sigma_Q$ and $b_R:= \angles{g}^\omega_R$ together with the Hardy--Littlewood maximal inequality.

Estimate \ref{temp:a3} implies  estimate \ref{temp:a4} via the substitutions $a_Q:=\sum_{S\supseteq Q} \tilde{a}_S$ and $b_R:=\sum_{S\supseteq R} \tilde{b}_R$. We next check that, conversely,  estimate \ref{temp:a4} implies estimateestimate \ref{temp:a3} via the substitutions 
$$
\tilde{a}_Q:= (\sup_{S\supseteq Q} a_S - \sup_{S\supseteq \hat{Q}} a_S) \quad\text{ and } \quad \tilde{b}_R:= (\sup_{S\supseteq R} b_S - \sup_{S\supseteq \hat{R}} b_S),
$$
where $\hat{Q}$ and $\hat{R}$ denote the dyadic parents of the cubes $Q$ and $R$. 

By the monotone converge theorem, we may assume without loss of generality that the the families $a$ and $b$ are supported on finitely many cubes. Now, in the expression appearing on the right-hand side of estimate \ref{temp:a4}, by a telescoping summation, we have
\begin{equation}\label{temp:a5}
\sum_Q \tilde{a}_Q 1_Q= \sup_Q \sum_{R\supseteq Q} \tilde{a}_R 1_Q=\sup_Q \big( \sum_{R\supseteq Q} (\sup_{S\supseteq R} a_S - \sup_{S\supseteq \hat{R}} a_S) \big)1_Q= \sup_Q \sup_{S\supseteq Q} a_S 1_Q = \sup_Q a_Q 1_Q,
\end{equation}
and, in the expression appearing on the left-hand side of estimate  \ref{temp:a4}, again by a telescoping summation, we have
\begin{equation}\label{temp:a6}
\big(\sum_{Q\supseteq P} \tilde{a}_Q\big)= \big(\sum_{Q\supseteq P} (\sup_{S\supseteq Q} a_S - \sup_{S\supseteq \hat{Q}} a_S)\big)= \sup_{S\supseteq P} a_S\geq a_P.
\end{equation}
Combining the inequalities \eqref{temp:a5} and \eqref{temp:a6} for the family $a$ and the same inequalities for the family $b$ with estimate \ref{temp:a4} yields estimate \ref{temp:a3}. The proof is complete.

\end{proof}

Similarly, using the same substitutions as in the proof of Lemma \ref{lemma:reformulations_summation}, we obtain the following reformulations of the two-weight norm inequality for maximal operators:
\begin{lemma}[Reformulations for maximal operators]\label{lemma:reformulations_supremum} Let  $q\in(0,\infty)$ and $p\in(1,\infty)$. Then the following estimates are equivalent:
\begin{enumerate}[label=(\roman*)]
\item We have
$$
\norm{\sup_P \lambda_P \angles{f}^\sigma_P 1_P}_{L^q(\omega)}\lesssim_{p} \norm{f}_{L^p(\sigma)}
$$
for all functions $f$.
\item We have
$$
\norm{\sup_P\lambda_P a_P 1_P}_{L^q(\omega)}\lesssim_{p} \norm{\sup_Q a_Q 1_Q }_{L^p(\sigma)}
$$
for all families $a$.
\item We have
$$
\norm{\sup_P \lambda_P  \big(\sum_{Q\supseteq P} \tilde{a}_Q\big) 1_P}_{L^q(\omega)}\lesssim_{p} \norm{\sum_Q \tilde{a}_Q 1_Q }_{L^p(\sigma)}
$$ 
 for all families $\tilde{a}$.
\end{enumerate}
\end{lemma}
\subsection{Characterization of multipliers between Carleson coefficients}We characterize the two-weight norm inequality for multipliers of Carleson coefficients. In addition to being interesting in its own right, this characterization is applied to characterize the two-weight norm inequality for summation operators under the $A_\infty$ assumption (see Proposition \ref{proposition:ainfty}).
\begin{proposition}[Characterization of multipliers of Carleson coefficients]\label{proposition:multipliers_carleson} Let $\sigma$ and $\omega$ be locally finite Borel measures. Let $\{\mu_Q\}_{Q \in \cd}$ be a family of non-negative reals (the multiplier of Carleson coefficients). Then the following assertions are equivalent:
\begin{enumerate}[label=(\roman*)]
\item \label{temp:1} We have
$$
\int \sup_{R\subseteq Q} \Big(\mu_R 1_R \Big) \domega \leq C \sigma(Q) \quad\text{for every cube $Q\in \cd$.}
$$
\item \label{temp:2} We have
\begin{equation*}
\norm{\{ \mu_Q a_Q\}}_{f^{1,\infty}(\omega)}\leq C \norm{a}_{f^{1,\infty}(\sigma)}\quad\text{for every family $\{a_Q\}$.}
\end{equation*}
\item \label{temp:3} We have
$$
\norm{\{\frac{\omega(Q)}{\sigma(Q)} \mu_Q b_Q\}}_{f^{\infty,1}(\sigma)}\leq C \norm{b}_{f^{\infty,1}(\omega)}\quad\text{for every family $\{b_Q\}$.}
$$
\end{enumerate}
Furthermore, the constants in the assertions are comparable.
\end{proposition}
\begin{proof} The equivalence of assertions \ref{temp:2} and \ref{temp:3} follows from the duality  in the discrete Littlewood--Paley spaces by using Proposition \ref{proposition:normduality}. The equivalence of  assertions \ref{temp:1} and \ref{temp:2} can be checked using essentially a standard proof of Sawyer's two-weight norm inequality for maximal operators. For the reader's convenience, we write out the proof. Assertion \ref{temp:2} implies  assertion \ref{temp:1} by substituting the family $\{a_R\}$ with $a_R=1$ when $R\subseteq Q$ and $a_R=0$ when $R\nsubseteq Q$. Assertion \ref{temp:1} implies  assertion \ref{temp:2} as follows.

 Let $a=\{a_Q\}$ be a family of non-negative reals. By the monotone convergence theorem, we may assume without loss of generality that the family $a$ is supported on finitely many cubes. 
We linearize the supremum (which now is a maximum) by writing
$$
\max_Q \mu_Q a_Q 1_Q = \sum_Q \mu_Q a_Q 1_{E(Q)}
$$
for the pairwise disjoint sets $E(Q)$, which can be defined, for example, as follows: We define $\tilde{E}(Q):=\{x\in Q: \max \mu_R a_R 1_R(x)= \mu_Q a_Q\}$ and $E(Q):= \tilde{E}(Q) \setminus \bigcup_{R\subsetneq Q }\tilde{E}(R)$. By using this linearization, we have
\begin{equation}\label{temp:4}
\int \sup_Q \mu_Q a_Q 1_Q \domega = \sum_Q \mu_Q a_Q \omega(E(Q))
\end{equation}
and hence we need to prove the estimate
$$
\sum_Q \mu_Q a_Q \omega(E(Q))\leq C \norm{\sup a_Q 1_Q }_{L^1(\sigma)}.
$$
By the dual estimate for the Carleson coefficients (see the remark after Proposition \ref{proposition:normduality}), this estimate
holds if and only if the Carleson condition
$$
\sum_{R\subseteq Q} \mu_R \omega(E(R))\leq C \sigma(Q)
$$
holds. This  condition holds because, by the assumption, we have
$$
\sum_{R\subseteq Q} \mu_R \omega(E(R)) \leq \int \sup_{R\subseteq Q} \mu_R 1_R \domega \leq C \sigma(Q).
$$
The proof of the equivalence of assertions \ref{temp:1} and \ref{temp:2} is complete.

\end{proof}
We recall that the {\it dyadic Fujii--Wilson $A_\infty$ characteristic} $[\sigma]_{A_\infty(\omega)}$ (of a measure $\sigma$ with respect to a measure $\omega$) is defined by 
$$[\sigma]_{A_\infty(\omega)}:= \sup_{Q\in\cd} \frac{1}{\sigma(Q)} \int \sup_{R\in\cd:R\subseteq Q} \Big( \frac{\sigma(R)}{\omega(R)} 1_R\Big) \domega.$$
Accordingly, the measure $\sigma$ is said to satisfy the $A_\infty$ condition with respect to the measure $\omega$ if
$$
[\sigma]_{A_\infty(\omega)}<\infty.
$$

Applying Proposition \ref{proposition:multipliers_carleson} to the family $\mu:= \{\frac{\sigma(Q)}{\omega(Q)}\}$ of multipliers, we record the following corollary:
\begin{corollary}[$A_\infty$ condition: Characterization in terms of Carleson condition]\label{corollary:a_infty_via_carleson}Let $\sigma$ and $\omega$ be locally finite Borel measures. Then the measure $\sigma$ satisfies the Fujii-Wilson $A_\infty$ condition with respect to the measure $\omega$ if and only if every family of coefficients that is $\omega$-Carleson is also $\sigma$-Carleson. Furthermore, quantitatively,
$$
\norm{ b}_{f^{\infty,1}(\sigma)}\leq [\sigma]_{A_\infty(\omega)} \norm{b}_{f^{\infty,1}(\omega)}\quad\text{for every family $b:=\{b_Q\}$.}
$$
\end{corollary}
Assume that the measure $\mu$ has no point masses. Under this assumption, the coefficients $\{b_Q\}$ are $\mu$-Carleson, which means (in our normalization) that
$$
\frac{1}{\mu(Q)}\sum_{R\in\cd: R\subseteq Q} b_R \mu(R)\leq C\quad \text{ for all dyadic cubes $Q$},
$$
if and only if they are {\it $\mu$-sparse}, which means that there exist pairwise disjoint sets $E_Q\subseteq Q$, $Q\in\cd$, such that $$b_R\leq C \frac{\mu(E_Q)}{\mu(Q)} \quad \text{ for all dyadic cubes $Q$}.$$ This equivalence was originally proven by Verbitsky \cite[Corollary 2]{verbitsky1996}. An alternative proof was given by Lerner and Nazarov \cite[Lemma 6.3]{lerner2015} (for the most important particular type of coefficients)  and Cascante and Ortega \cite[Theorem 4.3]{cascante2017} (for general type of coefficients). Furthermore, it was noticed by H\"anninen \cite{hanninen2017b} that the equivalence holds for not only dyadic cubes but for general sets  (for example, for dyadic rectangles). 

Combining the equivalence with Corollary \ref{corollary:a_infty_via_carleson} yields the following corollary:
\begin{corollary}[$A_\infty$ condition: Characterization in terms of disjoint sets]\label{corollary:a_infty_disjoint}Let $\sigma$ and $\omega$ be locally finite Borel measures. Assume that neither $\sigma$ nor $\omega$ has point masses. Then the measure $\sigma$ satisfies the Fujii-Wilson $A_\infty$ condition with respect to the measure $\omega$ if and only if the following holds: For each collection $\{F_Q\}$ of disjoint sets with $F_Q\subseteq Q$ there exists a collection $\{E_Q\}$ of disjoint sets with $E_Q\subseteq Q$ such that
$$
\frac{\omega(F_Q)}{\omega(Q)}\leq [\sigma]_{A_\infty(\omega)} \frac{\sigma(E_Q)}{\sigma(Q)}.
$$
\end{corollary}

\section{Proofs of results}

\subsection{Scale of conditions for maximal operators}
We recall that, for $\gamma\in(0,\infty)$, the auxiliary quantity $\Lambda^{\sup}_{\gamma,Q}$ is defined by $$\Lambda^{\sup}_{\gamma,Q}:= \left(\frac{1}{\omega(Q)}\int_Q ( \sup_{R\subseteq Q} \lambda_Q 1_Q )^\gamma 
\domega \right)^{\frac{1}{\gamma}}.$$
In this section, we prove the following result:
\begin{proposition}[Scale of conditions for maximal operators]
\label{proposition:scale_maximal_2}
The following assertions hold:
\begin{enumerate}[label=(\roman*)]
\item (Sufficient condition) We have
$$
\norm{M_\lambda(\cdotroomy)}_{L^p(\sigma)\to L^q(\omega)}\lesssim_{p,q} \big(\int \sup_Q \lambda_Q^q  \Big(\frac{\omega(Q)}{\sigma(Q)} \Big)^{\frac{q}{p-q}} (\Lambda_{q,Q}^{\sup})^{\frac{q^2}{p-q}} 1_Q \domega\big)^{\frac{p-q}{pq}}.
$$
\item (Necessary condition) Let $\gamma \in(0,q)$. Then we have
$$
 \left(\int \sup_Q \Big(\frac{\omega(Q)}{\sigma(Q)} \Big)^{\frac{q}{p-q}} (\Lambda^{\sup}_{\gamma,Q})^{\frac{pq}{p-q}} 1_Q \domega\right)^{\frac{p-q}{pq}} \lesssim_{\gamma,p,q} \norm{M_\lambda(\cdotroomy)}_{L^p(\sigma)\to L^q(\omega)}.
$$
\end{enumerate}
\end{proposition}
First, we prove two lemmas, which combined yield the necessary condition.

\begin{lemma}[Replacing the coefficients by their averages 
for maximal operators]\label{lemma:max_replacing}Let $p,q\in(0,\infty)$, and $\gamma \in (0,q)$. Then the following assertions are equivalent:

\begin{enumerate}[label=(\roman*)]
\item \label{temp:19} We have
\begin{equation}
\label{temp:eq:1}
\norm{\sup_Q \lambda_Q a_Q 1_Q}_{L^q(\omega)}\leq C \norm{\sup_Q a_Q 1_Q}_{L^p(\sigma)} \quad\text{for every family $a$.}
\end{equation}
\item  \label{temp:11} We have
$$
\norm{\sup_Q \Lambda_{\gamma,Q}^{\sup} a_Q 1_Q}_{L^q(\omega)}\lesssim_\gamma C \norm{\sup_Q a_Q 1_Q}_{L^p(\sigma)} \quad\text{for every family $a$.}
$$
\end{enumerate}
\begin{proof}Since $\lambda_Q\leq  \Lambda_{\gamma,Q}^{\sup} := (\angles{(\sup_{Q\subseteq S} \lambda_Q 1_Q)^\gamma}^\omega_S)^{\frac{1}{\gamma}}$,  assertion \ref{temp:11} implies  assertion \ref{temp:19} trivially. We next prove the converse. We substitute the monotonous rearrangement $\tilde{a}_Q:=\sup_{R\supseteq Q} a_R$ 
into estimate \eqref{temp:eq:1}. Under this substitution, the right-hand side of the estimate remains unchanged,
$$
\norm{\sup_Q \tilde{a}_Q 1_Q}_{L^p(\sigma)}=\norm{\sup_Q (\sup_{R\supseteq Q} a_Q) 1_Q}_{L^p(\sigma)}=\norm{\sup_Q a_Q 1_Q}_{L^p(\sigma)},
$$
and, by interchanging the order of the suprema, the left-hand side of the estimate becomes
$$
\norm{\sup_Q \lambda_Q \tilde{a}_Q 1_Q}_{L^q(\omega)}=\norm{\sup_Q \lambda_Q (\sup_{R\supseteq Q}a_R) 1_Q}_{L^q(\omega)}= \norm{\sup_R (\sup_{Q\subseteq R} \lambda_Q 1_Q) a_R }_{L^q(\omega)}.
$$
By the scaling of the $L^p$ norms, and by the Hardy--Littlewood maximal inequality, we estimate this from below as
\begin{equation*}
\begin{split}
&\norm{\sup_R (\sup_{Q\subseteq R} \lambda_Q 1_Q) a_R }_{L^q(\omega)}=\norm{ (\sup_R (\sup_{Q\subseteq R} \lambda_Q 1_Q) a_R )^\gamma}_{L^{\frac{q}{\gamma}}(\omega)}^{\frac{1}{\gamma}}\\
&\gtrsim_{\gamma} \norm{ \sup_S \angles{(\sup_R (\sup_{Q\subseteq R} \lambda_Q 1_Q) a_R )^\gamma}^\omega_S 1_S }_{L^{\frac{q}{\gamma}}(\omega)}^{\frac{1}{\gamma}}\geq \norm{ \sup_S (\angles{(\sup_{Q\subseteq S} \lambda_Q 1_Q)^\gamma}^\omega_S)^{\frac{1}{\gamma}} a_S 1_S }_{L^{q}(\omega)}\\
&=: \norm{ \sup_S  \Lambda_{\gamma,S}^{\sup} a_S 1_S }_{L^{q}(\omega)}.
\end{split}
\end{equation*}
The proof is complete.
\end{proof}
\end{lemma}
\begin{lemma}[Necessary condition 
for maximal operators]\label{lemma:max_necessary}We have
$$
\left(\int \sup_Q \lambda_Q^{\frac{pq}{p-q}} \Big(\frac{\omega(Q)}{\sigma(Q)}\Big)^{\frac{q}{p-q}} 1_Q \domega \right)^{\frac{p-q}{pq}}\lesssim_{p,q} \norm{M_\lambda(\cdotroomy)}_{L^p(\sigma)\to L^q(\omega)}.
$$
\end{lemma}
\begin{proof}
By the duality in the Littlewood--Paley spaces, the two-weight norm inequality for maximal operators is equivalent to the estimate
$$
\int \big(\sum_Q \lambda_Q^q \Big(\frac{\omega(Q)}{\sigma(Q)}\Big) b_Q 1_Q \big)^\frac{p}{p-q} \domega \leq C^\frac{pq}{p-q} \norm{b}_{f^{\infty,1}(\omega)}^\frac{p}{p-q}.
$$
By the comparison of the $\ell^p$ norms, we have
\begin{equation*}
\begin{split}
&\int \big(\sum_Q \lambda_Q^q \Big(\frac{\omega(Q)}{\sigma(Q)}\Big) b_Q 1_Q \big)^\frac{p}{p-q} \domega\geq \int \big(\sum_Q \lambda_Q^{\frac{pq}{p-q}} \Big(\frac{\omega(Q)}{\sigma(Q)}\Big)^{\frac{p}{p-q}} b_Q^{\frac{p}{p-q}} 1_Q \big)\domega\\
&=  \sum_Q  b_Q^{\frac{p}{p-q}} \lambda_Q^{\frac{pq}{p-q}} \Big(\frac{\omega(Q)}{\sigma(Q)}\Big)^{\frac{q}{p-q}} \omega(Q).
\end{split}
\end{equation*}
By the scaling of the discrete Littlewood--Paley norms, and by renaming $\tilde{b}:=b^\frac{p}{p-q}$,  we have $$\norm{b}_{f^{\infty,1}(\omega)}^\frac{p}{p-q}=\norm{b^\frac{p}{p-q}}_{f^{\infty,\frac{p-q}{p}}(\omega)}=\norm{\tilde{b}}_{f^{\infty,\frac{p-q}{p}}(\omega)}.$$ Therefore, altogether, we have
$$
\sum_Q  \tilde{b}_Q  \lambda_Q^{\frac{pq}{p-q}} \Big(\frac{\omega(Q)}{\sigma(Q)}\Big)^{\frac{q}{p-q}} \omega(Q)\leq C^\frac{pq}{p-q} \norm{\tilde{b}}_{f^{\infty,\frac{p-q}{p}}(\omega)}.
$$

The following duality formula obtained by Verbitsky \cite{verbitsky1996} holds: For every measure $\mu$ and every $s\in(0,1]$, we have
\begin{equation}
\label{eq:dualnormformula}
\norm{a}_{f^{1,\infty}(\mu)}\eqsim_s \sup \{\sum_Q a_Q b_Q \mu(Q)\colon \, \,  b\in f^{\infty,s}(\mu) \text{ with } \norm{b}_{f^{\infty,s}(\mu)}\leq 1\}.
\end{equation}
(For $s=1$, this is the usual duality in the discrete Littlewood--Paley spaces.) Applying this formula completes the proof.
\end{proof}

\begin{proof}[Proof of Proposition \ref{proposition:scale_maximal_2}]
We observe that the necessary condition follows by combining Lemma \ref{lemma:max_replacing} and Lemma \ref{lemma:max_necessary}. We next prove the sufficient condition.

By the duality in the Littlewood--Paley spaces, the  two-weight norm inequality 
\eqref{max} for the maximal operator is equivalent to the estimate
$$
\int \big(\sum_Q \lambda_Q^q \Big(\frac{\omega(Q)}{\sigma(Q)}\Big) b_Q 1_Q \big)^\frac{p}{p-q} \domega \leq C^\frac{pq}{p-q} \norm{b}_{f^{\infty,1}(\omega)}^\frac{p}{p-q}.
$$
By the using an equivalent expression (Lemma \ref{lemma:eq_exp}), we have
$$
\int \big(\sum_Q \lambda_Q^q \Big(\frac{\omega(Q)}{\sigma(Q)}\Big) b_Q 1_Q \big)^\frac{p}{p-q} \domega\eqsim_{p,q}
\sum_Q \lambda_Q^q b_Q \omega(Q) \Big( \frac{1}{\sigma(Q)} \sum_{R\subseteq Q} \lambda_R^q b_R \omega(R) \Big)^{\frac{q}{p-q}}.
$$ 
By using twice the dual estimate $\sum_Q c_Q d_Q \mu(Q)\leq \norm{c}_{f^{1,\infty}(\mu)} \norm{d}_{f^{\infty,1}(\mu)}$ for the discrete Littlewood--Paley spaces (see Proposition \ref{proposition:normduality}), we obtain
\begin{equation*}
\begin{split}
&\sum_Q \lambda_Q^q b_Q \omega(Q) \Big( \frac{1}{\sigma(Q)} \sum_{R\subseteq Q} \lambda_R^q b_R \omega(R) \Big)^{\frac{q}{p-q}}\\
&\leq \norm{b}_{f^{\infty,1}(\omega)}^{\frac{q}{p-q}} \sum_Q \lambda_Q^q b_Q \omega(Q) \Big( \frac{1}{\sigma(Q)}  \int (\sup_{R\subseteq Q} \lambda_R^q 1_R )\domega \Big)^{\frac{q}{p-q}}\\
&\leq \norm{b}_{f^{\infty,1}(\omega)}^{\frac{q}{p-q}} \norm{b}_{f^{\infty,1}(\omega)} \int \sup_Q \lambda_Q^q \Big( \frac{1}{\sigma(Q)}  \int (\sup_{R\subseteq Q} \lambda_R^q 1_R )\domega \Big)^{\frac{q}{p-q}} 1_Q \domega.
\end{split}
\end{equation*}
The proof is complete.
\end{proof}

\subsection{Characterization for summation operators under the $A_\infty$ assumption}Let $p\in(1,\infty)$, $q\in(0,\infty)$, and $q<p$. 
We recall that the integral expression $I_{\sigma,\omega,p,q,\lambda}$ is defined by
\begin{equation}
\label{temp:6}
I_{\sigma,\omega,p,q,\lambda}:= \left( \int \big(\sum_Q \lambda_Q \Big(\frac{\omega(Q)}{\sigma(Q)}\Big)^{\frac{1}{q}} 1_Q \big)^{\frac{pq}{p-q}} \dsigma \right)^\frac{p-q}{pq}.
\end{equation}
In this section, we prove the following result:
\begin{proposition}[Characterization under the $A_\infty$ assumption]\label{proposition:a_infty_2} Let $\sigma$ and $\omega$ be measures that satisfy the $A_\infty$ condition with respect to each other. Let $p\in(1,\infty)$ and $q\in(0,\infty)$ be such that $q<p$. Then we have the following characterization by subranges:
\begin{itemize}
\item In the subrange $q\in(0,1]$, we have
\begin{equation}
\begin{split}
&[\omega]_{A_\infty(\sigma)}^{-\frac{1-q}{q}} I_{\sigma,\omega,p,q,\lambda}\\
&\lesssim_{p,q} \norm{T_\lambda (\cdotroomy \sigma)}_{L^p(\sigma)\to L^q(\omega)} \\
& \lesssim_{p,q} [\sigma]_{A_\infty(\omega)}^{\frac{1-q}{q}} I_{\sigma,\omega,p,q,\lambda}.
\end{split}
\end{equation}

\item In the subrange $q\in(1,\infty)$, we have
\begin{equation}
\begin{split}
&\max\{ [\sigma]_{A_\infty(\omega)}^{-\frac{q-1}{q}} I_{\sigma,\omega,p,q,\lambda}, [\omega]_{A_\infty(\sigma)}^{-\frac{1}{p}} I^*_{\sigma,\omega,p,q,\lambda}\} \\
&\lesssim_{p,q} \norm{T_\lambda (\cdotroomy \sigma)}_{L^p(\sigma)\to L^q(\omega)} \\
& \lesssim_{p,q} \min\{[\omega]_{A_\infty(\sigma)}^{\frac{q-1}{q}} I_{\sigma,\omega,p,q,\lambda}, [\sigma]_{A_\infty(\omega)}^{\frac{1}{p}} I^*_{\sigma,\omega,p,q,\lambda}\}.
\end{split}
\end{equation}

\end{itemize}
\end{proposition}
\begin{remark}We recall that, in contrast to the subrange $q\in(0,1)$, in the subrange $q\in[1,\infty)$ the $L^q(\omega)-L^{q'}(\omega)$ duality is available and hence
$$
\norm{T_{\{\lambda_Q\}}(\cdotroomy \sigma)}_{L^p(\sigma)\to L^q(\omega)}=\norm{T_{\{\lambda_Q \frac{\omega(Q)}{\sigma(Q)}\}}(\cdotroomy \omega)}_{L^{q'}(\omega)\to L^{p'}(\sigma)}.
$$
Therefore, in this subrange, the characterization can be stated equivalently in terms of the dual integral expression $I^*_{\sigma,\omega,p,q,\lambda}$ defined by
$$
I^*_{\sigma,\omega,p,q,\lambda}:= \left( \int \big(\sum_Q \lambda_Q \Big(\frac{\omega(Q)}{\sigma(Q)}\Big)^{\frac{1}{p}} 1_Q \big)^{\frac{pq}{p-q}} \domega \right)^\frac{p-q}{pq}
$$
and hence we can include both of the expressions  $I_{\sigma,\omega,p,q,\lambda}$ and $I^*_{\sigma,\omega,p,q,\lambda}$ in the statement in this subrange.
\end{remark}
\begin{proof}[Proof of Proposition \ref{proposition:a_infty_2}] First, we consider the range $0<q\leq 1<p<\infty $. The case $q=1$ is trivial:  the two-weight norm inequality is written out as
$$
\int \big( \sum_Q \lambda_Q \frac{\omega(Q)}{\sigma(Q)} 1_Q \big) f\dsigma \leq C \norm{f}_{L^p(\sigma)},
$$
which by the $L^p(\sigma)-L^{p'}(\sigma)$ duality is equivalent to $\norm{\sum_Q \lambda_Q \frac{\omega(Q)}{\sigma(Q)} 1_Q}_{L^{p'}(\sigma)}\leq C$. We now assume that $q\in(0,1)$. We give a proof only for the estimate
$$
I_{\sigma,\omega,p,q,\lambda}\lesssim_{p,q} [\omega]_{A_\infty(\sigma)}^{\frac{1-q}{q}} \norm{T_\lambda (\cdotroomy \sigma)}_{L^p(\sigma)\to L^q(\omega)},
$$
since the reverse estimate can be proven in a similar way.

By duality in the discrete Littlewood--Paley norms, the two-weight norm estimate
$$
\norm{T_\lambda (f \sigma)}_{L^q(\omega)}\leq C \norm{f}_{L^p(\sigma)} \quad\text{for all functions $f$}
$$
is equivalent to the bilinear estimate
\begin{equation}\label{temp:7}
\sum_Q \lambda_Q^q a_Q b_Q \omega(Q) \lesssim_{p,q} C^q \norm{a}_{f^{\frac{p}{q},\infty}(\sigma)} \norm{b}_{f^{\infty,\frac{1}{1-q}}(\omega)} \quad\text{for all families $a$ and $b$.}
\end{equation}

By the scaling of the Littlewood--Paley norms, and by the $A_\infty$ assumption together with Proposition \ref{proposition:multipliers_carleson}, we have
\begin{equation}\label{temp:8}
\norm{b}_{f^{\infty,\frac{1}{1-q}}(\omega)}=\norm{b^{\frac{1}{1-q}}}_{f^{\infty,1}(\omega)}^{1-q}\leq [\omega]_{A_\infty(\sigma)}^{1-q} \norm{b^{\frac{1}{1-q}}}_{f^{\infty,1}(\sigma)}^{1-q}= [\omega]_{A_\infty(\sigma)}^{1-q}\norm{b}_{f^{\infty,\frac{1}{1-q}}(\sigma)}.
\end{equation}
Substituting this estimate \eqref{temp:8} into estimate \eqref{temp:7}, we obtain
\begin{equation}
\label{temp:9}\sum_Q \lambda_Q^q a_Q b_Q \omega(Q) \lesssim_{p,q} [\omega]_{A_\infty(\sigma)}^{1-q} C^q \norm{a}_{f^{\frac{p}{q},\infty}(\sigma)}\norm{b}_{f^{\infty,\frac{1}{1-q}}(\sigma)} \quad\text{for all $a$ and $b$.}
\end{equation}

By the factorization $f^{\frac{p}{q},\infty}(\sigma) \cdot f^{\infty,\frac{1}{1-q}}(\sigma)= f^{\frac{p}{q},\frac{1}{1-q}}(\sigma)$ (see Proposition \ref{proposition:factorization}), the following assertions hold: 
\begin{itemize}
\item For every $a$ and $b$, $$\norm{ab}_{f^{\frac{p}{q},\frac{1}{1-q}}(\sigma)}\lesssim_{p,q} \norm{a}_{f^{\frac{p}{q},\infty}(\sigma)}\norm{b}_{f^{\infty,\frac{1}{1-q}}(\sigma)}.$$
\item For every $c\in f^{\frac{p}{q},\frac{1}{1-q}}(\sigma)$ there exist $a\in f^{\frac{p}{q},\infty}(\sigma)$ and $b\in f^{\infty,\frac{1}{1-q}}(\sigma)$ such that $c=ab$ and 
$$ \norm{a}_{f^{\frac{p}{q},\infty}(\sigma)}\norm{b}_{f^{\infty,\frac{1}{1-q}}(\sigma)}\lesssim_{p,q} \norm{c}_{f^{\frac{p}{q},\frac{1}{1-q}}(\sigma)}.$$
\end{itemize} 
By these assertions, estimate \eqref{temp:8} is equivalent to the estimate
\begin{equation}
\label{temp:17}
\sum_Q \lambda_Q^q c_Q \omega(Q) \lesssim_{p,q} [\omega]_{A_\infty(\sigma)}^{1-q} C^q \norm{c}_{f^{\frac{p}{q},\frac{1}{1-q}}(\sigma)} \quad\text{for all families $c$}.
\end{equation}

By the duality $(f^{\frac{p}{q},\frac{1}{1-q}}(\sigma))^*= f^{\frac{p}{p-q},\frac{1}{q}}(\sigma)$ in the discrete Littlewood--Paley spaces, estimate \eqref{temp:17} is equivalent to the estimate
$$
\left( \int_Q \big( \lambda_Q  \Big( \frac{\omega(Q)}{\sigma(Q)}\Big)^{\frac{1}{q}} 1_Q \big)^\frac{pq}{p-q} \dsigma \right)^{\frac{p-q}{p}} \lesssim_{p,q} [\omega]_{A_\infty(\sigma)}^{1-q} C^q.
$$

Next, we consider the range $1<q<p<\infty$. We write the proof only for the estimate
$$ I_{\sigma,\omega,p,q,\lambda}\lesssim_{p,q}  [\sigma]_{A_\infty(\omega)}^{\frac{q-1}{q}}  \norm{T_\lambda (\cdotroomy \sigma)}_{L^p(\sigma)\to L^q(\omega)},
$$
as the reverse estimate and the dual estimates (with $q'$ and $p$, $\frac{\omega(Q)}{\sigma(Q)}\lambda_Q$ and $\lambda_Q$, and $\omega$ and $\sigma$ interchanged) can be proven similarly.

The two-weight norm inequality \eqref{sum} is equivalent  (see Lemma \ref{lemma:reformulations_summation} for the proof of this) to the bilinear estimate
\begin{equation}\label{temp:d1}
\sum_Q \lambda_Q a_Q b_Q \omega(Q)\leq C \norm{a}_{f^{p,\infty}(\sigma)}  \norm{b}_{f^{q',\infty}(\omega)}. 
\end{equation}
By the scaling of the Littlewood--Paley norms, and by the $A_\infty$ assumption together with Proposition \ref{proposition:multipliers_carleson}, we have
\begin{equation}\label{temp:d2}
\begin{split}
&\norm{b}_{f^{q',\infty}(\omega)}=
\norm{\{b_Q^{q'}\}}_{f^{1,\infty}(\omega)}^{\frac{1}{q'}}= \\
&\norm{\{b_Q^{q'} \frac{\omega(Q)}{\sigma(Q)}  \cdot \frac{\sigma(Q)}{\omega(Q)} \}}_{f^{1,\infty}(\omega)}^{\frac{1}{q'}}\\
&\leq [\sigma]_{A_\infty(\omega)}^{\frac{1}{q'}} \norm{\{ b_Q^{q'} \frac{\omega(Q)}{\sigma(Q)}\}}_{f^{1,\infty}(\sigma)}^{\frac{1}{q'}}= [\sigma]_{A_\infty(\omega)}^{\frac{1}{q'}} \norm{\{ b_Q \big(\frac{\omega(Q)}{\sigma(Q)}\big)^{\frac{1}{q'}}\}}_{f^{q',\infty}(\sigma)}.
\end{split}
\end{equation}
Combining estimate \eqref{temp:d2} with estimate \eqref{temp:d1} and writing $\tilde{b}_Q:=b_Q \big(\frac{\omega(Q)}{\sigma(Q)}\big)^{\frac{1}{q'}}$, we obtain
\begin{equation}
\label{temp:d3}
\sum_Q \lambda_Q a_Q \tilde{b}_Q \big( \frac{\omega(Q)}{\sigma(Q)}\big)^{\frac{1}{q}} \sigma(Q) \leq C  [\sigma]_{A_\infty(\omega)}^{\frac{1}{q'}} \norm{a}_{f^{p,\infty} (\sigma)}  \norm{\tilde{b}}_{f^{q',\infty}(\sigma)}.
\end{equation}

We define the exponent $r\in(1,\infty)$ by setting $\frac{1}{r'}:=\frac{1}{p}+\frac{1}{q'}$. Thus $r=\frac{pq}{p-q}$. By the factorization $f^{p,\infty}(\sigma) \cdot f^{q',\infty}(\sigma)= f^{r',\infty}(\sigma)$, the following assertions hold: 
\begin{itemize}
\item For every $a$ and $b$, $$\norm{ab}_{f^{r',\infty}(\sigma)}\lesssim_{p,q} \norm{a}_{f^{p,\infty}(\sigma)}\norm{b}_{f^{q',\infty}(\sigma)}.$$
\item For every $c\in f^{r',\infty}(\sigma)$ there exist $a\in f^{p,\infty}(\sigma)$ and $b\in f^{q',\infty }(\sigma)$ such that $c=ab$ and 
$$ \norm{a}_{f^{p,\infty}(\sigma)}\norm{b}_{f^{q',\infty}(\sigma)}\lesssim_{p,q} \norm{c}_{f^{r',\infty}(\sigma)}.$$
\end{itemize} 
By these assertions, estimate \eqref{temp:d3} is equivalent to the estimate
\begin{equation*}
\sum_Q \lambda_Q\big( \frac{\omega(Q)}{\sigma(Q)}\big)^{\frac{1}{q}} c_Q \sigma(Q)  \leq C  [\sigma]_{A_\infty(\omega)}^{\frac{1}{q'}}  \norm{c}_{f^{r',\infty}(\sigma)}.
\end{equation*}
By the $f^{r',\infty}(\sigma) - f^{r,1}(\sigma)$ duality, this estimate is equivalent to the estimate
\begin{equation*}
\norm{ \{ \lambda_Q \big( \frac{\omega(Q)}{\sigma(Q)}\big)^{\frac{1}{q}}\}}_{f^{r,\infty}(\sigma)} \leq C  [\sigma]_{A_\infty(\omega)}^{\frac{1}{q'}}.
\end{equation*}
The proof is complete.
\end{proof}

\subsection{Inequality for summation operators via  maximal operators}\label{subsection:proof_connection}We recall that the auxiliary quantity $\Lambda_Q=\Lambda^{\summation}_Q$ is defined by $$\Lambda^{\summation}_Q:= \frac{1}{\omega(Q)}\sum_{R\subseteq Q} \lambda_R \omega(R).$$
In this section, we prove the following result:
\begin{proposition}[Characterization 
for summation operators in terms of 
maximal operators] Let $1<q<p<\infty$. 
Then
\begin{equation*}
\begin{split}
&\norm{T_{\{\lambda_Q\}}(\cdotroomy \sigma )}_{L^p(\sigma)\to L^q(\omega)}\\
&\eqsim_{p,q} \norm{M_{\{\Lambda^{\summation}_Q\}}(\cdotroomy \sigma)}_{L^p(\sigma)\to L^q(\omega)} +  \norm{M_{\{\frac{\omega(Q)}{\sigma(Q)}\Lambda^{\summation}_Q \} }(\cdotroomy \omega)}_{L^{q'}(\omega)\to L^{p'}(\sigma)}.
\end{split}
\end{equation*} 

\end{proposition}
\begin{proof} By Lemma \ref{lemma:reformulations_summation}, the two-weight norm inequality
$$
\norm{\sum_P \lambda_P \angles{f}^\sigma_P 1_P }_{L^q(\omega)}\lesssim_{p,q} C \norm{f}_{L^p(\sigma)}
$$
 is equivalent to the estimate
\begin{equation}\label{temp:b3}
\sum_P \lambda_P \omega(P) \big(\sum_{Q\supseteq P} a_Q \big)\big(\sum_{R\supseteq P} b_R \big) \lesssim \norm{\sum_Q a_Q 1_Q}_{L^p(\sigma)}\norm{\sum_R b_R 1_R}_{L^{q'}(\omega)}.
\end{equation}

Since $Q\cap R\supseteq P$, by dyadic nestedness, we have either $R\subseteq Q$ or $Q\subseteq R$. Hence, the summation  splits into the cases $P \subseteq R \subseteq Q$ and  $P\subseteq Q \subseteq R$. Therefore,  estimate \eqref{temp:b3} is equivalent to the pair of estimates
\begin{subequations}
\begin{align}
\label{temp:b4}&  \sum_R b_R \big(\sum_{Q\supseteq R} a_R\big) \big( \sum_{P\subseteq R} \lambda_P \omega(P)  \big) \lesssim \norm{\sum_Q a_Q 1_Q}_{L^p(\sigma)}\norm{\sum_R b_R 1_R}_{L^{q'}(\omega)}, 
 \\
\label{temp:b5}& \sum_Q a_Q \big(\sum_{R\supseteq Q} b_R\big) \big( \sum_{P\subseteq Q} \lambda_P \omega(P) \big)\lesssim \norm{\sum_Q a_Q 1_Q}_{L^p(\sigma)}\norm{\sum_R b_R 1_R}_{L^{q'}(\omega)}.
  \end{align}
\end{subequations}

We handle only  subestimate \eqref{temp:b4}, as the other subestimate \eqref{temp:b5} can be handled similarly. By the $f^{q',1}(\omega)-f^{q,\infty}(\omega)$ duality in the discrete Littlewood--Paley spaces, subestimate \eqref{temp:b4} is equivalent to the estimate
$$
\norm{\sup_R \big(\sum_{Q\supseteq R} a_R\big) \big( \frac{1}{\omega(R)} \sum_{P\subseteq R} \lambda_P \omega(P) \big) }_{L^q(\omega)} \lesssim_{q} \norm{\sum_Q a_Q 1_Q}_{L^p(\sigma)}.
$$
By Lemma \ref{lemma:reformulations_supremum}, this estimate is equivalent to the two-weight norm inequality
$$
\norm{\sup_Q \Lambda_Q \angles{f}^\sigma_Q 1_Q}_{L^q(\omega)}\lesssim_{p,q}   \norm{f}_{L^p(\sigma)}.
$$
The proof is complete.
\end{proof}

\end{document}